\documentclass{amsart}
\usepackage{amssymb}
\usepackage{graphicx}
\usepackage{amscd}
\usepackage{comment}
\usepackage[all]{xy}
\usepackage{color}
\newtheorem{thm}{Theorem}[section]
\newtheorem{cor}[thm]{Corollary}
\newtheorem{lem}[thm]{Lemma}
\newtheorem{prop}[thm]{Proposition}

\theoremstyle{definition}
\newtheorem{defn}[thm]{Definition}

\theoremstyle{remark}
\newtheorem{rem}[thm]{Remark}
\newtheorem{ex}[thm]{Example}
\numberwithin{equation}{section}

\newcommand{\norm}[1]{\left\Vert#1\right\Vert}

\DeclareMathOperator{\im}{im}
\newcommand{\ol}{\overline}

\newcommand{\SFH}{\mathit{SFH}}

\newcommand{\HFLh}{\widehat{\mathit{HFL}}}
\newcommand{\HFKh}{\widehat{\mathit{HFK}}}
\newcommand{\HFKm}{\mathit{HFK}^-}

\newcommand{\EH}{\mathit{EH}}

\newcommand{\CFh}{\widehat{\mathit{CF}}}

\renewcommand{\hat}[1]{\widehat{#1}}
\def\a{\alpha}
\def\b{\beta}
\def\g{\gamma}
\def\d{\delta}

\def\S{\Sigma}

\def\t{\mathfrak{t}}

\def\bolda{\boldsymbol{\alpha}}
\def\boldb{\boldsymbol{\beta}}

\def\W{\mathcal{W}}

\def\X{\mathcal{X}}
\def\Int{\text{Int}}

\def\Z{\mathbb{Z}}
\def\Q{\mathbb{Q}}
\def\R{\mathbb{R}}

\def\T{\mathbb{T}}
\def\E{\mathcal{E}}

\def\FF{\mathbb{F}}
\def\CC{\mathcal{C}}

\def\NN{\mathbb{N}}

\def\L{\widehat{\mathfrak{L}}}
\def\Th{\widehat{\mathfrak{T}}}
\def\Lm{\mathfrak{L}^-}

\def\x{\mathbf{x}}
\def\y{\mathbf{y}}
\def\spinc{\text{Spin}^c}

\def\Id{\text{Id}}
\def\tb{\text{tb}}
\def\std{\text{std}}
\def\can{\text{can}}

\def\cL{\mathcal{L}}

\begin{document}

\title[Functoriality of the $\EH$ class under Lagrangian concordances]{Functoriality of the $\EH$ class
and the LOSS invariant under Lagrangian concordances}%

\author{Marco Golla}%
\address{Laboratoire de Math\'ematiques Jean Leray, 2, rue de la Houssini\`ere,
BP 92208, F-44322 Nantes Cedex 3, France}
\email{marco.golla@univ-nantes.fr}

\author{Andr\'as Juh\'asz}%
\address{Mathematical Institute, University of Oxford, Andrew Wiles Building,
Radcliffe Observatory Quarter, Woodstock Road, Oxford, OX2 6GG, UK}%
\email{juhasza@maths.ox.ac.uk}%

\thanks{Research supported by a Royal Society Research Fellowship. This project has received funding from the European Research Council (ERC) under the European Union's Horizon 2020 research and innovation programme (grant agreement No 674978).}


\begin{abstract}
We show that the $\EH$ class and the LOSS invariant of Legendrian knots in contact 3-manifolds
are functorial under regular Lagrangian concordances in Weinstein cobordisms.
This gives computable obstructions to the existence of regular Lagrangian concordances.
\end{abstract}

\maketitle
\section{Introduction}

Link Floer homology, defined by Ozsv\'ath and Szab\'o~\cite{OSz2} and denoted $\HFLh$,
is an invariant of links well-defined up to isomorphism.
In the case of knots, it is also known as knot Floer homology, and denoted $\HFKh$.
It becomes natural for \emph{decorated} links according to
the work of Dylan Thurston, Ian Zemke, and the second author~\cite{naturality}. Furthermore, the second author
showed~\cite{cob} that \emph{decorated} link cobordisms induce functorial cobordism maps on $\HFLh$.
As exhibited by Sarkar~\cite{basepoint} and Zemke~\cite{quasistab},
moving the decorations around the link often induces a non-trivial monodromy of link Floer homology.

Let $\Lambda$ be a Legendrian knot in a contact 3-manifold $(Y, \xi)$.
Then $\Lambda$ admits a Legendrian tubular neighborhood $N(\Lambda)$; i.e.,
a regular neighborhood identified with $S^1 \times \R^2$ such that $\Lambda = S^1 \times \{0\}$, and
\[
\xi|_{N(\Lambda)} = \ker(\cos \theta \,dx - \sin \theta \, dy),
\]
where $\theta$ is the angular coordinate on $S^1$, and $(x,y)$ are the Euclidean coordinates on $\R^2$;
see \cite[Example~2.5.10]{Geiges}.
If $M_\Lambda = Y \setminus N(\Lambda)$, then $\partial M_\Lambda$ is a convex surface with dividing set~$\g_\Lambda$
that consists of two curves, each corresponding to the framing $\tb(\Lambda)$.
We call $(M_\Lambda, \g_\Lambda)$ the sutured manifold \emph{complementary} to $\Lambda$.
We denote by $Y_\Lambda = Y_{\tb(\Lambda)}(\Lambda)$ the result of Dehn filling $M_\Lambda$ with slope $\g_\Lambda$,
and write $\Lambda'$ for the core of the Dehn filling. There is a pair of points $P_\Lambda \subset \Lambda'$
such that $P_\Lambda \times S^1 \subset \Lambda' \times D^2$ is glued to~$\g_\Lambda$.
Furthermore, $R_\pm(P)$ is the component of $\Lambda' \setminus P$ for which
$R_\pm(P) \times S^1$ is glued to~$R_\pm(\g_\Lambda)$. Then $(Y_\Lambda, \Lambda', P_\Lambda)$
is a decorated link in the sense of \cite[Definition~4.4]{cob}; see Section~\ref{sec:SFH}.

Stipsicz and V\'ertesi~\cite{Vertesi} defined the invariant
\[
\EH(\Lambda) := \EH(\xi|_{M_\Lambda}) \in \SFH(-M_\Lambda, -\g_\Lambda) \cong \HFKh(-Y_\Lambda, \Lambda', P_\Lambda)
\]
of Legendrian knots. Note that $\EH(\Lambda)$ is non-vanishing whenever $\EH(\xi) \neq 0$,
since the Honda--Kazez--Mati\'c gluing map~\cite{TQFT} for the sutured submanifold $(M_\Lambda, \g_\Lambda)$
of $(Y, \xi)$ takes $\EH(\Lambda)$ to $\EH(\xi)$.

\begin{defn} \label{def:XL}
  Suppose that $(X, \omega)$ is a Liouville cobordism from $(Y_-, \xi_-)$ to $(Y_+,\xi_+)$.
  Let $\Lambda_\pm$ be a Legendrian knot in $(Y_\pm, \xi_\pm)$, and let
  $L$ be a Lagrangian concordance in $(X, \omega)$ from $\Lambda_-$ to $\Lambda_+$.
  Then $X_L$ is the result of gluing $L \times D^2$ to $X \setminus N(L)$
  with framing $\tb(\Lambda_-) = \tb(\Lambda_+)$, and we write $L'$ for $L \times \{0\}$.
\end{defn}

Note that $(X_L, L')$ is a concordance from $(Y_{\Lambda_-}, \Lambda_-')$
to $(Y_{\Lambda_+}, \Lambda_+')$.
If $(X, \omega)$ is Weinstein, then Eliashberg, Ganatra, and Lazarev~\cite[Definition~2.1]{flexible}
say that the Lagrangian cobordism $L$ is \emph{regular} if the Liouville vector field can be
chosen to be tangent to $L$. For example, Conway, Etnyre, and Tosun~\cite[Lemma~3.4]{fillings} proved
that a Lagrangian cobordism in the symplectization $\left(\R \times Y, d(e^t \a)\right)$
of a contact manifold $(Y, \xi)$ with contact form $\a$ is regular if it is decomposable;
see Section~\ref{sec:background}. (Note that we had also observed this lemma, but decided not to
include it in this paper once \cite{fillings} appeared.)
We are now ready to state our first main result.

\begin{thm} \label{thm:main}
Suppose that $L$ is a regular Lagrangian concordance from $\Lambda_-$ to $\Lambda_+$
in the Weinstein cobordism $(X, \omega)$ from $(Y_-, \xi_-)$ to $(Y_+, \xi_+)$.
Choose an arbitrary decoration~$\sigma'$ on $L' \subset X_L$ consisting of two arcs, one of which
connects $R_-(P_{\Lambda_-})$ and $R_-(P_{\Lambda_+})$,
and the other $R_+(P_{\Lambda_-})$ and $R_+(P_{\Lambda_+})$.
We write $\ol{\cL'}$ for the reverse of the decorated concordance $\cL' = (X_L, L', \sigma')$
from $(-Y_{\Lambda_+}, \Lambda_+', P_{\Lambda_+})$ to $(-Y_{\Lambda_-}, \Lambda_-', P_{\Lambda_-})$.
Then the knot cobordism map
\[
F_{\ol{\cL'}} \colon \HFKh(-Y_{\Lambda_+}, \Lambda_+', P_{\Lambda_+}) \to \HFKh(-Y_{\Lambda_-}, \Lambda_-', P_{\Lambda_-})
\]
takes $\EH(\Lambda_+)$ to $\EH(\Lambda_-)$.
\end{thm}

It is natural to ask whether Theorem~\ref{thm:main} holds for exact Lagrangian concordances.
It would be true if every exact Lagrangian concordance were regular; see Eliashberg~\cite[Problem~5.1]{Wein}.

Given a Legendrian knot $\Lambda$ in a contact 3-manifold $(Y,\xi)$, Lisca, Ozsv\'ath, Stipsicz, and Szab\'o~\cite{LOSS}
defined an invariant $\L(\Lambda) \in \HFKh(-Y,\Lambda,\t_\xi)$, where $\t_\xi$ is the $\spinc$ structure defined by
the contact structure~$\xi$. This is now commonly known as the LOSS invariant.
In light of naturality, this is really an invariant of a Legendrian knot with a decoration~$P$
consisting of two points; i.e.,
\[
\L(\Lambda, P) \in \HFKh(-Y, \Lambda, P, \t_\xi).
\]

Baldwin, Vela-Vick, and V\'ertesi~\cite[p.~926]{equivalence} conjectured that the
LOSS invariant is functorial under Lagrangian cobordisms.
This is supported by a result of Baldwin and Sivek~\cite[Theorem~1.2]{monopole-loss},
which gives a contravariant map for every Lagrangian concordance
that preserves a monopole Floer Legendrian knot invariant. However, they do not show
that this map is an invariant of the Lagrangian cylinder.
Furthermore, they also proved \cite[Theorem~1.5]{iso} that there is a
(non-natural) isomorphism between knot Floer homology and monopole knot homology
that sends the LOSS invariant to their monopole Legendrian invariant.
It is also worth mentioning that it is currently not known whether
the cobordism maps in the Heegaard Floer setting agree with the maps in the monopole setting.

Stipsicz and V\'ertesi~\cite{Vertesi} showed that gluing a suitable basic slice to $\partial M_\Lambda$
induces a map
\[
F \colon \SFH(-M_\Lambda, -\g_\Lambda) \to \HFKh(-Y, \Lambda, P)
\]
that takes $\EH(\Lambda)$ to $\L(\Lambda, P)$.
As $F$ is not always invertible, $\L$ carries less information than $\EH$.
However, unlike $\EH$, the LOSS invariant can vanish even when $\EH(\xi) \neq 0$,
which can be used to obstruct the existence of regular Lagrangian concordances.
Our second result is that $\L$ is functorial under regular Lagrangian concordances.

\begin{thm}\label{thm:L}
Let $L$ be a regular Lagrangian concordance from $\Lambda_-$ to $\Lambda_+$
in the Weinstein cobordism $(X,\omega)$ from $(Y_-,\xi_-)$ to $(Y_+,\xi_+)$.
Choose arbitrary decorations $P_\pm$ on $\Lambda_\pm$ consisting of two points,
and a decoration $\sigma$ on $L$ consisting of two arcs, one of which
connects $R_-(P_-)$ and $R_-(P_+)$, and the other $R_+(P_-)$ and $R_+(P_+)$. Then
$\cL := \left(X, L, \sigma \right)$
is a decorated concordance
from $(Y_-, \Lambda_-, P_-)$ to $(Y_+, \Lambda_+, P_+)$ such that
\[
F_{\ol{\cL}} \left(\L(\Lambda_+, P_+)\right) = \L(\Lambda_-, P_-).
\]
\end{thm}

This implies the following, which also follows from the work of
Baldwin and Sivek~\cite{monopole-loss}.

\begin{cor} \label{cor:vanishing}
  Let $\Lambda_\pm$ be a Legendrian knot in $(Y_\pm, \xi_\pm)$.
  If $\L(\Lambda_+) = 0$ but $\L(\Lambda_-) \neq 0$, then there is no regular
  Lagrangian concordance from $\Lambda_-$ to $\Lambda_+$.
\end{cor}

Marengon and the second author~\cite[Theorem~1.2]{concordance} showed that,
given a decorated knot concordance $\CC$ in $I \times S^3$ from $(S^3, K_0, P_0)$ to $(S^3, K_1, P_1)$
and admissible diagrams $(\S_i, \bolda_i, \boldb_i, w_i, z_i)$ of $(S^3, K_i, P_i)$
for $i \in \{0, 1\}$, there is a filtered chain map
\[
f_\CC \colon \CFh(\S_0, \bolda_0, \boldb_0, w_0) \to \CFh(\S_1, \bolda_1, \boldb_1, w_1)
\]
of homological degree zero such that the induced morphism of spectral sequences
agrees with $F_\CC$ on the $E^1$ page and with $\Id_{\FF_2}$ on the total homology
and on the $E^\infty$ page. Consider the maps
\[
\delta_k \colon \HFKh_d(S^3, K, s) \to \HFKh_{d-1}(S^3, K, s-k)
\]
for $k \ge 1$. Using these, we have the following refinement
of Corollary~\ref{cor:vanishing}.

\begin{cor}
  Let $\Lambda_+$ and $\Lambda_-$ be Legendrian knots in $(S^3, \xi_\std)$.
  If $\delta_k\left(\L(\Lambda_+)\right) = 0$ but
  $\delta_k \left(\L(\Lambda_-)\right) \neq 0$, then there is no decomposable
  Lagrangian concordance from $\Lambda_-$ to $\Lambda_+$.
\end{cor}

Note that $\L$ is unchanged by positive stabilization, and hence gives rise to an
invariant $\Th$ of transverse knots. Baldwin, Vela-Vick,
and V\'ertesi~\cite[Theorem~1.1]{equivalence}
showed that $\Th$ agrees with $\hat{\theta}$ defined using
grid diagrams, which is algorithmically computable.
Ng, Ozsv\'ath, and Thurston~\cite[Theorem~1.1 and~1.2]{transverse}
proved that the mirrors of $10_{132}$ and $12n_{200}$ have Legendrian representatives
$\Lambda_1$ and $\Lambda_2$, both with $\tb = -1$ and $r = 0$, for which $\hat{\theta}(\Lambda_1) = 0$ and
$\hat{\theta}(\Lambda_2) \neq 0$. Furthermore, the pretzel knots $P(-4,-3,3)$ and
$P(-6,-3,3)$ have pairs of Legendrian representatives $\Lambda_1$ and $\Lambda_2$,
both with $\tb = -1$ and $r = 0$, for which $\hat{\theta}(\Lambda_1) \neq 0$ and $\hat{\theta}(\Lambda_2) \neq 0$,
but $\d_1 \circ \hat{\theta}(\Lambda_1) = 0$ while $\d_1 \circ \hat{\theta}(\Lambda_2) \neq 0$.
Hence, we obtain the following.

\begin{prop} \label{prop:nonexistence}
  Let
  \[
  K \in \{\, -10_{132}, -12n_{200}, P(-4,-3,3), P(-6,-3,3) \,\}.
  \]
  Then there exist Legendrian knots $\Lambda_1$ and $\Lambda_2$
  in $(S^3, \xi_\std)$ of topological type $K$
  with the same rotation and Thurston--Bennequin numbers such that there
  is no decomposable Lagrangian concordance from $\Lambda_2$ to $\Lambda_1$.
\end{prop}

The knot $P(-4,-3,3)$ is $\ol{10_{140}}$ and $P(-6, -3, 3)$ is $12n_{582}$;
see \cite[Remark~4.5]{obstruction}. Furthermore, the pairs of Legendrian
representatives $\Lambda_1$ and $\Lambda_2$
shown in \cite[Figures~4 and~5]{transverse} are Lagrangian slice; i.e.,
there are Lagrangian concordances from the unknot to $\Lambda_1$ and to $\Lambda_2$.
In fact, these are both decomposable.
It follows from Proposition~\ref{prop:nonexistence} that
there is no decomposable Lagrangian concordance from $\Lambda_2$ to the unknot.
Indeed, otherwise, composing this with the decomposable Lagrangian concordance from the unknot
to $\Lambda_1$, we would get a decomposable Lagrangian concordance from $\Lambda_2$ to $\Lambda_1$.
More generally, Cornwell, Ng, and Sivek~\cite[Theorem~3.2]{obstruction} showed
that there is no decomposable Lagrangian concordance from a Legendrian knot to
any Legendrian unknot.

Let $K$ be a framed knot in $Y$ with longitude~$\ell$ and meridian~$m$.
Then let $M = Y \setminus N(K)$ be the knot
complement, and suppose that $\g_n$ is a dividing set on $\partial M$ consisting of
two parallel curves of slope $\ell + n \cdot m$.
Gluing a basic slice to $\partial M \times I$ with dividing set $\g_n$ on $\partial M \times \{0\}$
and $\g_{n+1}$ on $\partial M \times \{1\}$ induces a map
\[
\sigma_+ \colon \SFH(-M,-\g_n) \to \SFH(-M,-\g_{n+1}).
\]
We will prove the following result in Section~\ref{sec:proofs}.

\begin{cor} \label{cor:power}
 Let $\Lambda_\pm$ be a Legendrian knot in $(Y_\pm, \xi_\pm)$.
 If there is a regular Lagrangian concordance from $\Lambda_-$ to $\Lambda_+$ and
 $\L(\Lambda_+) \in \im(\sigma_+^k)$ for some $k \in \NN$, then
 $\L(\Lambda_-) \in \im(\sigma_+^k)$.
\end{cor}

Let $K$ be a framed null-homologous knot in $Y$.
The first author~\cite{comparing} and Etnyre, Vela-Vick, and Zarev~\cite{limit}
showed that the limit of $\SFH(M, \g_n)$ along the maps~$\sigma_+$ is
$\HFKm(-Y, K)$. Furthermore, if $K$ is Legendrian, then $\EH(K)$
limits to the LOSS invariant $\Lm(K)$. Then Proposition~\ref{prop:stab}
implies that every parametrized concordance (see Section~\ref{sec:param}) induces a map on the limit.
This construction is due to Tovstopyat-Nelip~\cite{limitconcordance} when $Y = S^3$.
It is an interesting question whether this coincides with the cobordism map defined
by Zemke on $\HFKm$.

\begin{prop}
  Suppose that $K_\pm$ is a null-homologous knot in $Y_\pm$, and
  let $C$ be a parametrized concordance from $K_-$ to $K_+$.
  Then its reverse induces a map
  \[
  F^-_{\ol{C}} \colon \HFKm(-Y_+, K_+) \to \HFKm(-Y_-, K_-).
  \]
  Furthermore, if $\xi_\pm$ is a contact structure on $Y_\pm$,
  and $\Lambda_\pm$ is a Legendrian representative of $K_\pm$,
  while $C$ is a regular Lagrangian concordance, then
  \[
  F^-_{\ol{C}} \left( \Lm(\Lambda_+) \right) = \Lm(\Lambda_-).
  \]
\end{prop}

Note that Sarkar~\cite{basepoint} and Zemke~\cite{quasistab} proved that the basepoint moving map is trivial
on $\HFKm$. The analogue of Corollary~\ref{cor:power} is the following:

\begin{cor}
 Let $\Lambda_\pm$ be a null-homologous Legendrian knot in $(Y_\pm, \xi_\pm)$.
 If there is a regular Lagrangian concordance from $\Lambda_-$ to $\Lambda_+$ and
 $\Lm(\Lambda_+) \in \im(U^k)$ for some $k \in \NN$, then
 $\Lm(\Lambda_-) \in \im(U^k)$.
\end{cor}

\subsection*{Acknowledgment} We would like to thank Sungkyung Kang, Steven Sivek,
And\-r\'as Stipsicz, and Ian Zemke for helpful discussions.

\section{Background} \label{sec:background}

\subsection{Lagrangian cobordisms}
An oriented link $\Lambda$ in the contact 3-manifold
$(Y,\xi)$ is \emph{Legendrian} if it is everywhere tangent to~$\xi$.
Let $(Y_-, \xi_-)$ and $(Y_+, \xi_+)$ be contact 3-manifolds.
Then a \emph{Weinstein cobordism} from $(Y_-, \xi_-)$ to $(Y_+, \xi_+)$
is a quadruple $(X, \omega, v, \phi)$, where
\begin{enumerate}
  \item X is a cobordism from $Y_-$ to $Y_+$,
  \item $\omega$ is a symplectic form on $X$,
  \item $v$ is a Liouville vector field on $X$ (i.e., $L_v \omega = \omega$) that points
  into $X$ along $Y_-$ and points out of $X$ along $Y_+$,
  \item $\lambda = \iota_v \omega$ induces $\xi_\pm$ on $Y_\pm$, and
  \item $\phi \colon X \to \R$ is a Morse function such that $d\phi(v) \ge c\norm{v}^2$ for some $c \in \R_+$
  and some Riemannian metric on $X$.
\end{enumerate}
Weinstein cobordism are precisely the ones that can be obtained by attaching so called Weinstein handles
of indices 0, 1, and 2 to $Y_-$, where the 2-handles are attached along Legendrian knots~$\Lambda$ with framing
$\tb(\Lambda) - 1$.

Suppose now that $\Lambda_\pm$ is a Legendrian link in $(Y_\pm, \xi_\pm)$.
According to Eliashberg, Ganatra, and Lazarev~\cite{flexible},
an \emph{exact Lagrangian cobordism} from $\Lambda_-$ to $\Lambda_+$
consists of a Weinstein cobordism $(X, \omega, v, \phi)$ from $(Y_-, \xi_-)$
to $(Y_+, \xi_+)$, together with a Lagrangian $L \subset X$ such that
\begin{enumerate}
  \item $L$ is exact; i.e., $\lambda|_L$ is exact,
  \item $v$ is tangent to $L$ near $\partial L$, and
  \item $\partial L = -\Lambda_- \cup \Lambda_+$.
\end{enumerate}
The exact Lagrangian cobordism $L$ is \emph{regular} if it is tangent to~$v$;
see \cite[Definition~2.1]{flexible}. In this paper, we say that $L$ is a \emph{Lagrangian concordance}
if it is a cylinder (this notion is more general than what is usually understood by concordance,
in that we do not require the ambient manifold to be the symplectization of a contact manifold).

Let $\a_0 = dz - y dx$ be the standard contact form and $\xi_0 = \ker(\a_0)$
the standard contact structure on $\R^3$.
Note that $S^3$ with the standard contact structure $\xi_\std = \ker(\a_\std)$
is the one-point compactification of $(\R^3,\xi_0)$.
Chantraine~\cite{Chantraine1} originally introduced the notion of \emph{Lagrangian cobordism}
in the symplectization $(\R \times \R^3, d(e^t \alpha_0))$ between Legendrian links
in $(\R^3, \xi_0)$. In the rest of this subsection, we focus on this special case,
as it admits a nice class of regular Lagrangian cobordisms.

\begin{defn}
Let~$\Lambda_-$ and~$\Lambda_+$ be Legendrian links in $(\R^3, \xi_0)$.
Then a Lagrangian cobordism from~$\Lambda_-$ to~$\Lambda_+$ consists of an embedded Lagrangian submanifold~$L$
of the symplectization $\left(\R \times \R^3, d(e^t\alpha_0)\right)$ such that, for some $T \in \R_+$,
\[
\begin{split}
\E_+(L,T) &:= L \cap \left((T,\infty) \times \R^3\right) = (T,\infty) \times \Lambda_+, \\
\E_-(L,T) &:= L \cap \left((-\infty, -T) \times \R^3\right) = (-\infty, -T) \times \Lambda_-,
\end{split}
\]
and $L_T := L \setminus (\E_+(L,T) \cup \E_-(L,T))$ is compact with boundary $\Lambda_+ \times \{T\} \sqcup -\Lambda_- \times \{-T\}$.
The Lagrangian cobordism $L$ is \emph{exact} if there is a function
$f \in C^\infty(L,\R)$ such that $df = e^t\a_0|_L$, and $f$ is constant on $\E_+(L)$ and $\E_-(L)$;
see Ekholm, Honda, and K\'alm\'an~\cite[Definition~1.1]{EHK}.
\end{defn}

According to Chantraine~\cite[Theorem~1.3]{Chantraine1}, if $\Lambda_-$ and $\Lambda_+$ are knots and
$L$ is a Lagrangian cobordism from $\Lambda_-$ to $\Lambda_+$,
then $r(\Lambda_-) = r(\Lambda_+)$ and $\tb(\Lambda_+) - \tb(\Lambda_-) = -\chi(L)$, where
$r$ is the rotation number and $\tb$ is the Thurston--Bennequin number.

If two Legendrian knots are Legendrian isotopic, then they are Lagrangian concordant.
Every Lagrangian concordance is automatically exact.
Chantraine~\cite{Chantraine2} showed that Lagrangian concordance is \emph{not} a symmetric relation.

Chantraine~\cite[Definition~1.4]{Chantraine3} and Ekholm, Honda, and K\'alm\'an~\cite[Section~6]{EHK} defined the class of
\emph{decomposable} Lagrangian cobordisms. These are products of certain elementary cobordisms,
namely Legendrian Reidemeister moves, saddles, and births of unknot components; see Figure~\ref{fig:decomposable}.
Each decomposable Lagrangian cobordism is exact, but it is not known whether the converse holds;
see~\cite[Question~1.5]{Chantraine3}, \cite[Question~8.10]{EHK}, and~\cite[Problem~5.1]{Wein}.
For a potential example of a non-decomposable Lagrangian concordance,
see Cornwell, Ng, and Sivek~\cite[Conjecture~3.3]{obstruction}.
Conway, Etnyre, and Tosun~\cite[Section~3.2]{fillings} extended the notion of decomposable
Lagrangian cobordism to symplectizations of arbitrary contact 3-manifolds, and they
proved~\cite[Lemma~3.4]{fillings} that the decomposable Lagrangian cobordisms in a symplectization
are all regular.

\begin{figure}
  \centering
  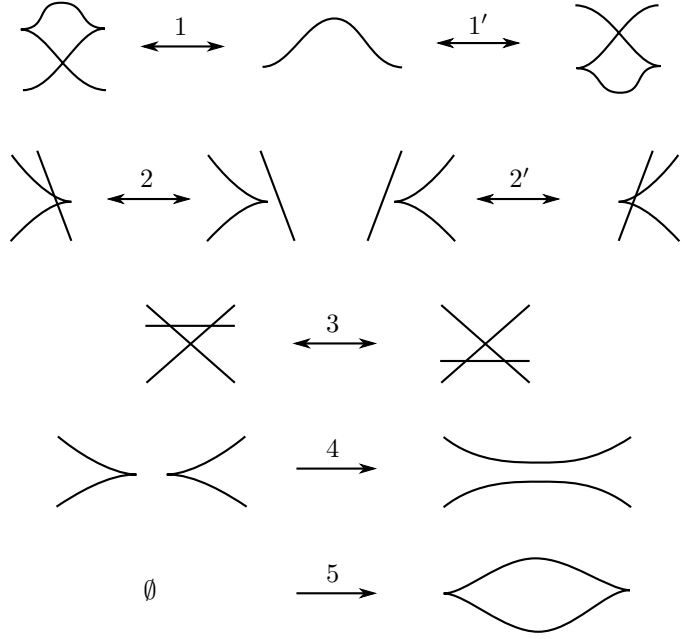
  \caption{Front projections (i.e., projections to the $(x, z)$ plane) of elementary Lagrangian cobordisms.
  Moves $1$, $1'$, $2$, $2'$, and $3$ are the Legendrian Reidemeister moves,
  $4$ is a saddle move, and $5$ is a birth.
  A Lagrangian cobordism is decomposable if it is a composition of elementary ones.}\label{fig:decomposable}
\end{figure}

\subsection{Sutured Floer homology} \label{sec:SFH}
Sutured Floer homology, defined by the second author~\cite{sutured}, assigns a finite-dimensional $\FF_2$-vector
space to a balanced sutured manifold. It is an extension of the hat version of Heegaard Floer homology,
due to Ozsv\'ath and Szab\'o~\cite{Osz8, OSz}, to 3-manifolds with boundary.
If $\xi$ is a contact structure on $M$ such that $\partial M$ is convex with dividing set~$\g$,
then Honda, Kazez, and Mati\'c~\cite{contact} assign to $\xi$ an element $\EH(\xi) \in \SFH(-M,-\g)$.

The second author~\cite[Definitions~2.3 and~2.4]{cob} introduced the notion of
sutured manifold cobordism.

\begin{defn}
Let $(M_-,\g_-)$ and $(M_+,\g_+)$ be balanced sutured manifolds.
A sutured manifold cobordism from $(M_-,\g_-)$ to $(M_+,\g_+)$
is a triple $\W = (W,Z,[\xi])$ such that
\begin{enumerate}
\item $W$ is a compact oriented 4-manifold with boundary and corners along $\partial Z$,
\item $Z$ is a compact codimension zero submanifold of $\partial W$ with $\partial W \setminus \Int(Z) = -M_- \sqcup M_+$, and
\item $[\xi]$ is the equivalence class of a positive contact structure on $Z$ such that
$\partial M_\pm$ is a convex surface with dividing set $\g_\pm$.
\end{enumerate}

We say that $\W$ is \emph{special} if $Z = I \times (\partial M_-)$
and $\xi$ is an $I$-invariant contact structure such that $\{t\} \times \partial M_-$
is a convex surface with dividing set $\{t\} \times \g_-$ for every $t \in I$ with
respect to the contact vector field $\partial / \partial t$.
\end{defn}

In~\cite{cob}, he also constructed functorial maps for sutured manifold cobordisms. In particular, every sutured
manifold cobordism $\W$ from $(M_-,\g_-)$ to $(M_+,\g_+)$ induces a linear map
\[
F_\W \colon \SFH(M_-,\g_-) \to \SFH(M_+,\g_+).
\]
This is a composition $F_{\W^s} \circ \Phi_{-\xi}$, where $-\xi$ is $\xi$ with its co-orientation reversed, and
\[
\Phi_{-\xi} \colon \SFH(M_-,\g_-) \to \SFH(M_- \cup -Z,\g_+)
\]
is the Honda--Kazez--Mati\'c gluing map~\cite{TQFT} induced by $-\xi$
and the sutured submanifold $(-M_-,-\g_-)$ of $(-M_- \cup Z,-\g_+)$.
We can now view $W$ as a cobordism $\W^s$ from $(M_- \cup -Z,\g_+)$ to $(M_+,\g_+)$
that is a product along the boundary; i.e., as a special cobordism.
It induces a map $F_{\W^s}$ by composing maps assigned to 4-dimensional handle attachments,
defined analogously to the Ozsv\'ath--Szab\'o handle maps~\cite{OSz10}.

A decorated knot is a triple $(Y, K, P)$, where $K$ is a knot in the oriented 3-manifold $Y$,
and $P$ is a pair of points on $K$ that divide it into arcs $R_-(P)$ and $R_+(P)$.
A decorated concordance $\X = (X, C, \sigma)$ from $(Y_0, K_0, P_0)$ to $(Y_1, K_1, P_1)$
consists of a cobordism $X$ from $Y_0$ to $Y_1$, a properly embedded annulus $C$ in $X$
with $\partial C = -K_0 \cup K_1$, and a pair of arcs $\sigma \subset C$ such that
one arc connects $R_+(P_0)$ with $R_+(P_1)$, and the other $R_-(P_0)$ with $R_-(P_1)$;
see \cite[Definition~4.5]{cob}. The sutured manifold cobordism $\W(\X) = (W, Z, [\xi])$
complementary to the concordance $\X$ is defined by taking $W = X \setminus N(C)$, and
$Z$ to be the unit normal circle bundle of~$C$, oriented as $\partial W$, together with the $S^1$-invariant
contact structure~$\xi$ that induces the dividing set $\sigma$ on $C$.
We define the map $F_\X$ induced by the decorated concordance $\X$ to be $F_{\W(\X)}$.

If $\W$ is a sutured manifold cobordism from $(M_-,\g_-)$ to $(M_+,\g_+)$, then
we can also view it as a cobordism $\ol{\W} = (W, Z, [-\xi])$ from $(-M_+,-\g_+)$ to $(-M_-,-\g_-)$.
In this paper, we call $\ol{\W}$ the \emph{reverse} of $\W$, as we think of our cobordisms going from
left to right. The second author \cite[Theorem~11.8]{cob} showed that $F_{\ol{\W}} = F_{\W}^*$
when $\W$ is a special cobordism, where the dual $F_{\W}^*$ is taken with respect to the natural pairing
\[
\langle \, , \, \rangle \colon \SFH(M,\g) \otimes \SFH(-M,-\g) \to \FF_2,
\]
obtained by choosing a diagram $(\S, \bolda, \boldb)$ of $(M,\g)$ and $(-\S, \bolda,\boldb)$
of $(-M, -\g)$, and defining $\langle\, \x , \y \,\rangle = \d_{\x, \y}$ for $\x$, $\y \in \T_\a \cap \T_\b$.
The sutured manifold cobordism $\W(\X)$ complementary to a concordance is equivalent to a special cobordism,
hence $F_{\ol{\X}} = F_{\X}^*$. This applies to the maps featuring in Theorems~\ref{thm:main} and~\ref{thm:L}.

\section{Parametrized concordances} \label{sec:param}

\begin{defn}
  Let $K_i$ be a knot in the oriented 3-manifold $Y_i$.
  A \emph{parametrized concordance} from $K_0$ to $K_1$ is a properly embedded annulus
  $C$ in a cobordism $X$ from $Y_0$ to $Y_1$,
  together with a normal framing and an identification $C \approx I \times S^1$ such that
  $\{i\} \times S^1 = \{i\} \times K_i$ for $i \in \{0 , 1\}$.
\end{defn}

\begin{rem}
By framing, we mean a concrete identification of a regular neighborhood $N(C)$ of $C$
with $C \times D^2$.
The normal framing of $C$ restricts to framings of $K_0$ and $K_1$.
\end{rem}

Given a parametrized concordance $(X, C)$ from $(Y_0, K_0)$ to $(Y_1, K_1)$, let
\[
\begin{split}
W_C &= X \setminus N(C), \\
M_i &= Y_i \setminus N(C) \text{ for } i \in \{0, 1\} \text{, and} \\
Z_C &= C \times S^1, \text{ oriented as } \partial W_C.
\end{split}
\]
If $\gamma$ is a dividing set on $S^1 \times S^1$ consisting of two parallel curves,
then let $\xi_\g$ be the I-invariant contact structure on $Z \approx I \times S^1 \times S^1$
such that $\{t\} \times S^1 \times S^1$ is convex with dividing set $\g_t = \{t\} \times \g$ for every $t \in I$.

\begin{ex} \label{ex:framing-dec}
  Suppose that $(X, C,\sigma)$ is a decorated concordance from $(Y_0, K_0, P_0)$ to $(Y_1, K_1, P_1)$.
  If we choose the identification $C \cong S^1 \times I$ such that
  $\sigma = (I \times \{x\}) \cup (I \times \{y\})$ for some $x$, $y \in S^1$ and
  $P_i = \{i\} \times P$ for $i \in \{0, 1\}$, then the $S^1$-invariant contact structure $\xi$ on
  $Z \approx C \times S^1$ with dividing set $\sigma$ on $C$
  is $I$-invariant as well. Hence $\xi = \xi_\g$ for $\g = P \times S^1$.
\end{ex}

As usual, we identify $\Q P^1$ with $\Q \cup \{\infty\}$ by mapping $[p : q]$ to $q/p$ if $p \neq 0$,
and to $\infty$ otherwise.
For $r = [p : q] \in \Q P^1$ and $p$, $q \in \Z$ relatively prime, let
\[
\ell_r := \left\{\, \left(e^{2\pi ipt}, e^{2\pi iqt}\right) \,\colon\, t \in I \,\right\}
\subset S^1 \times S^1,
\]
and write
\[
\g_r = \ell_r \cup -(e^{\pi i /q}, 1) \ell_r
\]
when $q \neq 0$, and $\g_r = \ell_r \cup -(1, -1) \ell_r$ otherwise.
Then $\g_r$ is a dividing set on $T^2 = S^1 \times S^1$.

Let $K$ be a framed knot in an oriented 3-manifold $Y$; i.e., we are given an identification
$N(K) \approx S^1 \times D^2$. Then we can talk about the sutured manifold
\[
Y(K,r) := \left(Y \setminus \Int(N(K)), \g_r \right)
\]
for any $r \in \Q P^1$.
For example, if $(K, P)$ is a decorated knot, $|P| = 2$, and the framing
$N(K) \approx S^1 \times D^2$ is chosen such that $P = \{\, (1,0), (-1,0)\,\}$,
then $Y(K,P) = Y(K, 0)$, as $\g_0$ consists of two meridional sutures
over~$P$.

If $(X, C)$ is a parametrized concordance from $(Y_0, K_0)$ to $(Y_1, K_1)$ and $r \in \Q P^1$,
then we denote the contact structure $\xi_{\g_r}$ by $\xi_r$, and we let
\[
F_{C,r} := F_{(W_C, Z_C, [\xi_r])} \colon \SFH(Y_0(K_0,r)) \to \SFH(Y_1(K_1,r)),
\]
where the framings of $K_0$ and $K_1$ are the restrictions of the normal framing of $C$.

\begin{defn}
  Suppose that $(X,C)$ is a parametrized concordance from $(Y_0, K_0)$ to $(Y_1, K_1)$  and $r \in \Q P^1$.
  We denote by $(X_r(C), C_r)$ the parametrized concordance where $X_r(C)$ is obtained by
  gluing $N(C)$ to $X \setminus N(C)$ along a map that sends each meridian $\{t\} \times \g_\infty$
  to $\{t\} \times \g_r$ for $t \in I$, and $C_r = C \times \{0\} \subset C \times D^2 \approx N(C)$.
  This carries a natural parametrization.
  We call $(X_r(C), C_r)$ the result of \emph{$r$-surgery along the concordance $C$}.
\end{defn}

Note that Definition~\ref{def:XL} is a special instance of the above construction.
In particular, $(X_L, L') = (X_n(L), L_n)$, where $n$ is $\tb(\Lambda_\pm)$ measured
with respect to the normal framing of $L$.

\begin{lem} \label{lem:surgery}
  Let $(X,C)$ be a parametrized concordance from $(Y_0, K_0)$ to $(Y_1, K_1)$  and $r \in \Q P^1$.
  If we choose a decoration $\sigma_r$ of the surgered concordance $(X_r(C), C_r)$ as in Example~\ref{ex:framing-dec},
  and the complementary sutured manifold cobordism
  \[
  \W(X_r(C), C_r, \sigma_r) = (W, Z, [\xi]),
  \]
  then $[\xi] = [\xi_r]$.
\end{lem}

\begin{proof}
  This is a straightforward consequence of the definitions.
\end{proof}

For $(r, r') \in \Q P^1 \times \Q P^1$, let
\[
\g_{r,r'} := \g_r \times \{0\} \cup \g_{r'} \times \{1\} \subset \partial \left(T^2 \times I\right).
\]
Then  $\left(T^2 \times I, \g_{r,r'}\right)$ is a balanced sutured manifold.
If $\zeta$ is a contact structure on $(T^2 \times I, \g_{r,r'})$, then it induces a gluing map
\[
\Phi_{K,\zeta} \colon \SFH(-Y(K,r)) \to \SFH(-Y(K,r')).
\]
Let $\ol{C}$ be the reverse of the concordance $C$.

\begin{prop} \label{prop:commut}
  Let $(X, C)$ be a parametrized concordance from $(Y_0, K_0)$ to $(Y_1, K_1)$.
  If $(r,r') \in \Q P^1 \times \Q P^1$ and $\zeta$
  is a contact structure on $(T^2 \times I, \g_{r,r'})$,
  then the following diagram is commutative:
  \[
  \xymatrix{
  \SFH\left(-Y_1(K_1, r)\right) \ar[r]^-{\Phi_{K_1, \zeta}} \ar[d]^-{F_{\ol{C},r}} &
  \SFH\left(-Y_1(K_1,r')\right) \ar[d]^{F_{\ol{C},r'}} \\
  \SFH\left(-Y_0(K_0, r)\right)  \ar[r]^-{\Phi_{K_0, \zeta}} & \SFH\left(-Y_0(K_0, r')\right).}
  \]
\end{prop}

\begin{proof}
  The map $F_{\ol{C},r}$ is defined by gluing $(Z_{\ol{C}}, \xi_r)$ to $Y_1(K_1, r)$, and then
  attaching 4-dimensional handles. As $Z_{\ol{C}}$ is just a collar of the boundary of the result of the gluing,
  we can assume all the handles are attached away from $Z_{\ol{C}}$. Then $\Phi_{K_0, \zeta}$
  is induced by gluing $(T^2 \times I, \zeta)$ to $Y_0(K_0, r)$.
  By \cite[Proposition~11.5]{cob}, disjoint gluing and handle maps commute.
  Hence, we also obtain $\Phi_{K_0, \zeta} \circ F_{\ol{C},r}$ if we first glue
  $\xi_r \cup \zeta$ to $Y_1(K_1, r)$, followed by the handle attachments along $Y_1(K_1, r)$.
  Since $\xi_r$ and $\xi_{r'}$ are $I$-invariant, $\xi_r \cup \zeta$ is isotopic to $\zeta \cup \xi_{r'}$.
  Gluing $\zeta$ to $Y_1(K_1)$ induces $\Phi_{K_1, \zeta}$, and then gluing $\xi_{r'}$
  to $Y_1(K_1, r')$ and attaching the handles induce $F_{\ol{C},r'}$. The result follows.
\end{proof}

\begin{prop} \label{prop:stab}
  Let $(X, C)$ be a parametrized concordance from $(Y_-, K_-)$ to $(Y_+, K_+)$, and let $n \in \Z$. Then
  \[
  \sigma_+ \circ F_{\ol{C},n} = F_{\ol{C},n+1} \circ \sigma_+.
  \]
\end{prop}

\begin{proof}
  For $n \in \Z$, we denote by $\zeta_n$ the basic slice on
  $(T^2 \times I, \g_{n, n + 1})$ corresponding to positive stabilization.
  Given a knot $K$ in $Y$, we let
  \[
  \sigma_+ \colon \SFH(-Y(K, n)) \to \SFH(-Y(K, n+1))
  \]
  be the map induced by gluing $\zeta_n$. Then this commutes with
  parametrized concordance maps by Proposition~\ref{prop:commut}.
\end{proof}

Let $\Lambda$ be a Legendrian knot in $(Y, \xi)$.
Using the above notation, the map
\[
F \colon \SFH(-M_\Lambda, -\g_\Lambda) \to \HFKh(-Y, \Lambda)
\]
of Stipsicz and V\'ertesi~\cite[Theorem~4.2]{Vertesi} that takes
$\EH(\Lambda)$ to $\L(\Lambda)$ is defined the following way.
Choose an identification $N(\Lambda) \approx S^1 \times D^2$,
and let $n \in \Q P^1$ be $\tb(\Lambda)$ with respect to this framing.
The framing determines the decoration $P = \{(1,0), (-1,0)\}$ of $\Lambda$.
Let $\zeta$ be the basic slice on $(T^2 \times I, \g_{n, 0})$
determined by a bypass arc oriented coherently with the meridian of~$\Lambda$
(this is given by the orientation of $\Lambda$). Then
\[
F = \Phi_{\Lambda,\zeta} \colon \SFH\left(-Y(\Lambda, n)\right) \to
\SFH(-Y(\Lambda, 0)) \cong \HFKh\left(-Y, \Lambda, P\right).
\]
We denote the image of $\EH(\Lambda)$ by $\L(\Lambda, P)$.

\section{Proofs} \label{sec:proofs}

\begin{defn} \label{def:Lag}
Let $\Lambda_\pm$ be a Legendrian knot in $(Y_\pm, \xi_\pm)$, and write
\[
(M_\pm,\g_\pm) := (M_{\Lambda_\pm},\g_{\Lambda_\pm}).
\]
Given a parametrized Lagrangian concordance $L$ in the Weinstein cobordism
$(X, \omega)$ from $\Lambda_-$ to $\Lambda_+$,
we can associate to it a special cobordism $\W(L) = (W_L, Z_L, [\xi_L])$
from $(M_-, \g_-)$ to $(M_+, \g_+)$, as follows.
Let $N(L)$ be a Lagrangian tubular neighborhood of $L$.
Then $N(L)$ is identified with the unit cotangent disk bundle $DT^*L$. We denote by $\partial_h N(L)$
the part of $\partial N(L)$ identified with the unit cotangent circle bundle $ST^*L$.
This carries the canonical contact structure $\xi_\can$, which has dividing set
$\g_\pm$ on $\partial M_\pm$.
Then the parametrization of $L$ induces a contactomorphism
\[
\varphi \colon \left(M_-, \g_-, \xi_-|_{M_-}\right) \to \left(M_- \cup \partial_h N(L), \g_-,
\xi_-|_{M_-} \cup \xi_\can \right).
\]
We then set
\[
W_L = (X \setminus N(L)) \cup (I \times M_+),
\]
and $Z_L = I \times \partial M_+$,
together with the $\R$-invariant contact structure $\xi_L$ such that
$\{t\} \times \partial M_+$ is a convex surface with dividing set $\{t\} \times \g_+$ for every $t \in I$.
We identify the incoming end of $\W(L)$ with $(M_-,\g_-)$ using $\varphi$.
We call $\W(L)$ the sutured manifold cobordism \emph{complementary} to the
parametrized Lagrangian cobordism~$L$.
\end{defn}

\begin{rem}
  It is important to note that $\xi_\can$ is a positive contact structure if we orient $\partial_h N(L)$
  as the boundary of~$N(L)$, which is the opposite of the boundary orientation of $X \setminus N(L)$.
  Hence, the contact structure $\xi_\can$ is not isotopic to any $S^1$-invariant contact structure
  arising from a decoration of the $\tb(\Lambda_\pm)$-surgered concordance $L'$.
\end{rem}

A sutured manifold cobordism
$\W = (W, Z, [\xi])$ from the contact manifold with convex boundary $(M_-,\g_-,\zeta_-)$ to $(M_+,\g_+,\zeta_+)$
is Weinstein essentially if its special part from $(M_-,\zeta_-)$ to $(M_+ \cup Z, \zeta_+ \cup \xi)$
can be obtained by attaching Weinstein 1- and 2-handles; see \cite[Remark~11.23]{cob}.

\begin{proof}[Proof of Theorem~\ref{thm:main}]
Choose a parametrization of $L'$ compatible with the decoration $\sigma'$ as in Example~\ref{ex:framing-dec}.
We endow $L$ with the parametrization induced by the $\tb(\Lambda_\pm)$-framed surgery.
Let $\W(L) = (W_L, Z_L, [\xi_L])$ be the special cobordism complementary to $L$, as in Definition~\ref{def:Lag}.
By Lemma~\ref{lem:surgery}, the sutured manifold cobordisms $\W(\cL')$ and $\W(L)$ from $(M_-, \g_-)$
to $(M_+, \g_+)$ are equivalent. Hence $F_{\ol{\cL'}} = F_{\ol{\W(\cL')}} = F_{\ol{\W(L)}}$.

Using Lemma~2.2 and the preceding discussion of
Eliashberg, Ganatra, and Lazarev~\cite{flexible}, we obtain that $W_L$ can be built
by attaching Weinstein handles to
$(M_- \cup \partial_h N(L), \g_+, \xi_-|_{M_-} \cup \xi_\can)$.
Hence, by \cite[Remark~11.23]{cob}, the sutured manifold cobordism
$\W(L)$ is Weinstein in the sense of \cite[Definition~11.22]{cob}.
The result now follows from a result of the second author~\cite[Theorem~11.24]{cob}
that the reverse of a Weinstein sutured manifold cobordism induces a map that preserve the $\EH$ class.
\end{proof}

\begin{proof}[Proof of Theorem~\ref{thm:L}]
 Let $L$ be a regular Lagrangian concordance from $\Lambda_-$
 to $\Lambda_+$. Choose a parametrization of $L$ compatible with the
 decoration~$\sigma$ as in Example~\ref{ex:framing-dec},
 and let $r$ be $\tb(\Lambda_-) = \tb(\Lambda_+)$ measured with respect to this normal framing.
 If we apply Proposition~\ref{prop:commut} with $C = L$, $K_0 = \Lambda_-$,
 $K_1 = \Lambda_+$, $r' = 0$,
 and $\zeta$ the basic slice on $(T^2 \times I, \g_{r, 0})$ used in the definition of the
 Stipsicz--V\'ertesi map, then we obtain the following commutative diagram:
 \[
  \xymatrix{
  \SFH\left(-Y_+(\Lambda_+, r)\right) \ar[r]^-{\Phi_{\Lambda_+, \zeta}} \ar[d]^-{F_{\ol{L},r}} & \SFH\left(-Y_+(\Lambda_+ ,0)\right) \ar[d]^{F_{\ol{L},0}} \\
  \SFH\left(-Y_-(\Lambda_-, r)\right)  \ar[r]^-{\Phi_{\Lambda_-, \zeta}} &
  \SFH\left(-Y_-(\Lambda_-, 0)\right).}
  \]
  Note that $\SFH\left(-Y_\pm(\Lambda_\pm,0)\right) \cong
  \HFKh\left(-Y_\pm, \Lambda_\pm, P_\pm\right)$ tautologically.
  If we endow $L'$ with the parametrization induced by $r$-framed surgery along $L$
  and denote by $\cL'$ the corresponding decorated concordance, then
  $F_{\ol{L}, r} = F_{\ol{\cL'}}$ by Lemma~\ref{lem:surgery}. Furthermore, $F_{\ol{L}, 0} = F_{\ol{\cL}}$.
  Hence, we can rewrite the above diagram as
  \begin{equation}\label{eqn:commut}
  \xymatrix{
  \SFH\left(-Y_+(\Lambda_+, r)\right) \ar[r]^-{\Phi_{\Lambda_+, \zeta}} \ar[d]^-{F_{\ol{\cL'}}} & \HFKh\left(-Y_+, \Lambda_+, P_+\right) \ar[d]^{F_{\ol{\cL}}} \\
  \SFH\left(-Y_-(\Lambda_-, r)\right)  \ar[r]^-{\Phi_{\Lambda_-, \zeta}} &
  \HFKh\left(-Y_-, \Lambda_-, P_-\right).}
  \end{equation}
  Now consider $\EH(\Lambda_+) \in \SFH\left(-Y_+(\Lambda_+, r)\right)$.
  By Theorem~\ref{thm:main}, we have
  \[
  F_{\ol{\cL'}}(\EH(\Lambda_+)) = \EH(\Lambda_-).
  \]
  Since $\Phi_{\Lambda_\pm, \zeta}\left(\EH(\Lambda_\pm)\right) = \L(\Lambda_\pm, P_\pm)$,
  the commutativity of diagram~\eqref{eqn:commut} amounts to
  $F_{L, \sigma}\left(\L(\Lambda_+, P_+)\right) = \L(\Lambda_-, P_-)$.
\end{proof}

\begin{proof}[Proof of Corollary~\ref{cor:power}]
  Let $L$ be a Lagrangian concordance from $\Lambda_-$ to $\Lambda_+$.
  Choose a parametrization of $L$, and let $n$ be $\tb(\Lambda_+) = \tb(\Lambda_-)$
  with respect to the normal framing of $L$.
  Then $F_{\ol{L},n} = F_{\ol{\cL'}}$ by Lemma~\ref{lem:surgery}. The result now follows from
  Theorem~\ref{thm:L} and Proposition~\ref{prop:stab}.
\end{proof}

\bibliographystyle{amsplain}
\bibliography{topology}
\end{document}